\documentclass{amsart}
\usepackage{graphicx}
\usepackage{amssymb}
\usepackage{epstopdf}
\usepackage{amsmath,amsthm,amscd,amssymb}
\usepackage{latexsym}
\usepackage[colorlinks,citecolor=red,pagebackref,hypertexnames=false]{hyperref}
\usepackage{geometry}                
\numberwithin{equation}{section}

\theoremstyle{plain}
\newtheorem{theorem}{Theorem}[section]
\newtheorem{lemma}[theorem]{Lemma}

\theoremstyle{definition}
\newtheorem{definition}[theorem]{Definition}

 \newtheorem*{problem}{Problem}
\newtheorem*{question}{Question}
\newtheorem{example}[theorem]{Example}

\newtheorem{case[theorem]}{Case}

\theoremstyle{remark}

\numberwithin{equation}{section}

\begin{document}

\title{On   Gabor orthonormal bases over finite prime fields}


\author{A. Iosevich, M. Kolountzakis, Yu. Lyubarskii, A. Mayeli and J. Pakianathan}

\date{\today}

\address{Alex Iosevich, Department of Mathematics, University of Rochester, Rochester, NY,USA}
\email{iosevich@math.rochester.edu}
\address{Mihalis Kolountzakis,, Department of Mathematics and Applied Mathematics, 
University of Crete, Crete
Greece}
\email{kolount@gmail.com}
\address{Yu. Lyubarskii, Department of Mathematical Sciences, Norwegian University of Science and Technology, Trohdheim, Norway} 
\email{yurii.lyubarskii@math.ntnu.no}
\address{Azita Mayeli, Department of Mathematics and Computer Science, Queensborough and The Graduate Center, City University of New York, USA} 
\email{AMayeli@gc.cuny.edu}
\address{Jonathan Pakianathan, Department of Mathematics, University of Rochester, Rochester, NY, USA}
\email{jonpak@math.rochester.edu}

\thanks{The work of the first and fifth listed authors was partially supported by the NSA Grant H98230-15-1-0319. The work of the second author has been supported by grant No 4725 of the University of Crete. The work of fourth author was partially supported by the PSC-CUNY Grant  B $\sharp$  69625-00 47.}

\maketitle

\begin{abstract} We study   Gabor orthonormal  windows in $L^2({\Bbb Z}_p^d)$ for translation and modulation sets $A$ and $B$, respectively, where $p$ is prime and $d\geq 2$. We prove that for a set $E\subset \Bbb Z_p^d$, the indicator function  $1_E$  is a Gabor window if and only if  $E$ tiles and is spectral. Moreover, we prove that for any function  $g:\Bbb Z_p^d\to \Bbb C$ with support $E$, 
if   the size of $E$ coincides with the size of the modulation set $B$ or if $g$ is positive,  then  $g$ is a unimodular function, i.e., $|g|=c1_E$, for some constant $c>0$,  and $E$ tiles and is spectral.  We also prove the existence of a Gabor window $g$ with  full support where neither $|g|$ nor $|\hat g|$ is an indicator function and $|B|<<p^d$. We conclude the paper with an example and open questions.\end{abstract} 


\section{Introduction}

\vskip.125in

In the classical case $\Bbb R^d$, it is possible to decompose a square integrable function in the Euclidean setting into the infinite sum of translations and modulations of a given function. The function is called {\it window} and the family of its translations and modulations is called {\it Gabor system}. The Gabor system is  a Gabor orthonormal basis if the system is an orthonormal basis. Gabor bases are strong tools in time-frequency analysis of signals (functions).  In this paper, we introduce the notion of Gabor orthonormal bases for the finite dimensional   vector spaces over finite fields and study their primary properties when the field is a prime filed. 

\vskip.125in

Let $A, B \subset {\Bbb Z}_p^d$, where ${\Bbb Z}_p$  is the cyclic group of size $p$ and ${\Bbb Z}^d_p$ is the $d$-dimensional vector space over 
${\Bbb Z}_p$. The purpose of this paper is to investigate {\it Gabor orthonormal bases}, also known as {\it Weyl-Heisenberg orthonormal bases} for $L^2({\Bbb Z}_p^d)$. These are the bases of the  form 
\begin{align}\label{Gabor System} 
\mathcal G(g, A, B):={\{g(x-a) \chi(x \cdot b)\}}_{a \in A; b \in B} ,
\end{align}
 with {\it window function}  $g: {\Bbb Z}_p^d \to {\mathbb C}$, $\|g\|_2=1$  (with respect to the counting measure) and $\chi$ is  any non-trivial character of the additive group  ${\Bbb Z}_p$.  The orthogonality is given by 
$$ \sum_{x \in {\Bbb Z}_p^d} g(x-a) \overline{g(x-a')} \chi(x \cdot (b-b'))=0 \ \text{if} \ (a,b) \not=(a',b').$$ 
 
   For convenience we work with $\chi(t)=e^{\frac{2 \pi i t}{p}}$, for $p$ prime, but any non-trivial character will do. The sets  $A$  and $B$ are called translation and modulation sets, respectively. 

Our investigation is partially motivated by the authors' previous work (\cite{IMP15}) on the Fuglede conjecture in ${\Bbb Z}_p^2$, $p$ prime. 

\begin{definition} We say that $E \subseteq {\Bbb Z}_p^d$ is spectral if there exists  $B\subseteq \Bbb Z_p^d$ such that  $L^2(E)$ possesses an orthogonal basis of exponentials ${\{\chi(x \cdot b) \}}_{b \in B}$. In this case we say that $(E,B)$ is a {\it spectral pair}.  (Note that this is a symmetric relation.)
\end{definition} 

\begin{definition} We say that $E \subseteq {\Bbb Z}_p^d$ tiles by translation if there exists $A \subseteq {\Bbb Z}_p^d$ such that 
$$ \sum_{a \in A} 1_E(x-a)=1 \ \text{for all} \ x \in {\Bbb Z}_p^d.$$ In this case we say that $(E,A)$ is a {\it tiling pair}. (Again, we note that the tiling pair property is also a symmetric relation.) The set $E$ is packing with $A$ if 
$$ \sum_{a \in A} 1_E(x-a)\leq 1 \ \text{for all} \ x \in {\Bbb Z}_p^d.$$
\end{definition} 

In \cite{IMP15} we proved that The Fuglede Conjecture holds for $\Bbb Z_p^2$, $p$ prime. Indeed:

\begin{theorem}\label{hottheorem} Let $E \subseteq {\Bbb Z}_p^2$, $p$ prime. Then $L^2(E)$ has an orthogonal basis of characters if and only if $E$ tiles 
${\Bbb Z}_p^2$ by translation. \end{theorem} 

 The Fuglede Conjecture was originally stated in $\Bbb R^d$ (\cite{Fu74}). 
For dimensions $d=3$ and  higher the conclusion of Theorem \ref{hottheorem} is known to be false in $\Bbb R^d$ (\cite{T04,KM06}).  
In ${\Bbb Z}_p^d$ the conclusion of Theorem \ref{hottheorem} is known to be false in dimensions four and higher.  As of dimension $3$,  the implication tile $\rightarrow$ spectral always holds  over the finite fields of prime order. It is also well-known that the Fuglede conjecture holds in $\Bbb Z_2^3$ and $\Bbb Z_3^3$ (\cite{Iosevich and CO}). It is trivial that the Fuglede conjecture holds in $\Bbb Z_p$, $p$ prime. 

 \medskip

 Our main results follow:

\begin{theorem}\label{basistilingtheorem} Suppose that $E\subseteq \Bbb Z_p^d$. Then 
  $ {\{|E|^{-1/2}1_E(x-a) \chi(x \cdot b) \}}_{a \in A; b \in B}$ is an orthonormal basis for $L^2(\Bbb Z_p^d)$ if and only if $(E, B)$ is a spectral pair and $(E, A)$ is a tiling pair.  
\end{theorem}

Our next two results on the   characterization of  Gabor orthonormal bases in $\Bbb Z_p^d$ are  motivated by  Liu-Wang's conjecture for the  Gabor orthonormal bases in $\Bbb R^d$ (\cite{Liu-Wang}).

\begin{theorem}\label{Liu-Wang_conjecture} Let $g\in L^2(\Bbb Z_p^d)$ and let  $A\subseteq \Bbb Z_p^d$ and $B\subseteq \Bbb Z_p^d$. Let   $E:= {\text supp}(g)$ and     $|E|=|B|$.  
Then 
 $\mathcal G(g, A, B)= {\{g(x-a) \chi(x \cdot b) \}}_{a \in A; b \in B}$ is  orthonormal  and complete in  $L^2(\Bbb Z_p^d)$ if and only if the following hold.

 \begin{itemize}
\item[(i)]    $|g| = |E|^{-1/2}1_E$. 

\vskip.125in

\item[(ii)] $(E, B)$ is a spectral pair. 

\vskip.125in 

\item[(iii)] $(E,A)$ is a  tiling pair.  

 \end{itemize}
 
Moreover,   if $d=2$ , $1<|A|, |B| <p^2$ and 
$ \mathcal G(g, A, B)$  is an orthonormal  basis for $L^2(\Bbb Z_p^2)$, then $E$ is graph of a function. 
\end{theorem} 

\medskip  

We say a set $E\subset \Bbb Z_p^2$ is {\it graph} of a function if there is a map 
$u:\Bbb Z_p\to \Bbb Z_p$ such that $E=\{(x,u(x)): \ u\in \Bbb Z_p\}$. 

\medskip

The following Theorem is similar to Theorem \ref{Liu-Wang_conjecture} but the size assumption on $B$ has been replaced by a positivity assumption.  

 \begin{theorem}\label{positive window} Let  $g\in L^2(\Bbb Z_p^d)$  and let  $A\subseteq \Bbb Z_p^d$ and $B\subseteq \Bbb Z_p^d$. Let   $E:= {\text supp}(g)$ and $g$ be positive on  $E$. 
Then 
 $\mathcal G(g, A, B)$ is an orthonormal  basis for $L^2(\Bbb Z_p^d)$ if and only if the following hold true:

 \begin{itemize}
\item[(i)]    $g= |E|^{-1/2}1_E$. 

\vskip.125in

\item[(ii)]   $(E, B)$ is a spectral pair. 

\vskip.125in 

\item[(iii)] $(E,A)$ is a  tiling pair.  
 \end{itemize}
\end{theorem} 

The following result   proves  that  the conclusions $(i)-(iii)$  in Theorems  \ref{Liu-Wang_conjecture} and \ref{positive window} generally fail if we remove either  the  size  or the  positivity assumption.  (See also Example \ref{example_of_graph}.)

Here and throughout, given $g: {\Bbb Z}_p^d \to {\Bbb C}$, the Fourier transform $\hat g$ is defined by 
\begin{equation} \label{ftdef} \widehat{g}(m)=p^{-d} \sum_{x \in {\Bbb Z}_p^d} \chi(-x \cdot m) g(x),  \end{equation} 
and Plancherel identity is given by 
\begin{align}\label{Plancherel identity} 
\sum_{m\in \Bbb Z_p^d} |\widehat g(m)|^2 = p^{-d}\sum_{x\in \Bbb Z_p^d} |g(x)|^2 .
\end{align}

\begin{theorem}\label{counter example} There exists an orthonormal basis $\mathcal G(g, A, B)$ with a Gabor window $g$ in $L^2(\Bbb Z_p^d)$,  where $g$ has none of the following properties: 
\vskip.125in 

\begin{itemize} 
\item (1) $|g|$ is  a  constant multiple of an indicator function of a set,  
\item (2) $|\hat g|$ is a  constant multiple of an  indicator function of a set. 
\end{itemize} 
\end{theorem} 

\vskip.125in 

One of the themes of this paper is characterization of window functions, which is one of the subtle problems in the study of  Gabor bases. There are only few instances where the characterization can be achieved. The following example is one of those.  

\medskip 
 
Assume that $A$ is a subspace of $\Bbb Z_p^d$. Then for some $1\leq k< d$, $A$ is isomorphic with $\Bbb Z_p^k\times \{{\bf 0}_{d-k}\}$, where ${\bf 0}_{d-k}$ is the vector zero  in $\Bbb Z_p^{d-k}$. Let $B$ be the orthogonal complement of $A$. Then $B$ is isomorphic to $\{{\bf 0}_{k}\} \times \Bbb Z_p^{d-k}$. For any $x\in \Bbb Z_p^d$, we denote $x=(x_1,x_2)\in \Bbb Z_p^{k} \times \Bbb Z_p^{d-k} $ where $x_1$ and $x_2$ are  partitions of $x$ in $\Bbb Z_p^d$.   
We  have the following result. 
 
\begin{theorem}\label{Characterization of Gabor window for lattices}
 Let $g\in L^2(\Bbb Z_p^d)$ and $A$ and $B$ be as above. Then the Gabor family  $\mathcal G(g, A,B)$ is an orthonormal  basis for $L^2(\Bbb Z_p^d)$ if and only if the following conditions hold true for any $1\leq k\leq d$: 

\begin{itemize} 
\item[(a)] $\left|\sum_{x_1\in \Bbb Z_p^k} g(x_1,x_2) \chi(-x_1\cdot m)\right| = \text{constant} \quad \forall\  x_2\in \Bbb Z_p^{d-k}$
\item[(b)] $\sum_{x_1\in \Bbb Z_p^k}  |g|^2(x_1,x_2)   = \text{constant}.$
\end{itemize} 
\end{theorem}

\section{Preliminaries and basic properties of Gabor orthonormal  bases in ${\Bbb Z}_p^d$}

In this section we collect the basic properties of   Gabor orthonormal bases in ${\Bbb Z}_p^d$. 

\begin{lemma}\label{basiccount}  The following hold true. 
\begin{itemize} 
\item[(i)] If the Gabor family (\ref{Gabor System}) is an orthonormal basis, then 
 $|A||B| = p^d$.
 \item[(ii)]  If $|A|=1$ and $B={\Bbb Z}_p^d$, then for any function    $g$   such that $|g(x)|=p^{-d/2}$,     (\ref{Gabor System}) is an orthonormal basis. 
 \item[(iii)] If  $|B|=1$ and $A={\Bbb Z}_p^d$, then   for any   function  $g$  such that $|\widehat{g}(m)|=p^{-d/2}$,   (\ref{Gabor System}) is an orthonormal basis. 
 \end{itemize} 
  \end{lemma} 

\begin{proof} The proof is immediate  from a direct calculation and basic linear algebra. 
\end{proof} 

\medskip 
 
 {\bf Remark.} Notice Lemma \ref{basiccount} provides a complete characterization of Gabor orthonormal bases of form (\ref{Gabor System})  over the one dimensional space $\Bbb Z_p$.

\medskip 

 \begin{lemma}\label{Fourier transform of G-OGB} Suppose that $g\in L^2(\Bbb Z_p^d)$   and $A, B\subseteq \Bbb Z_p^d$. Then  $\mathcal G(g, A, B)$ is an orthonormal  basis for  $L^2(\Bbb Z_p^d)$ if and only if $\mathcal G(\widehat g, B, -A)$ is an orthonormal basis for $L^2(\Bbb Z_p^d)$. 
\end{lemma}

 In the following sections we shall use the above results to construct large families of Gabor orthonormal  bases and make connections between this problem and the Fuglede conjecture\rq{}s result  in ${\Bbb Z}_p^2$  in Theorem \ref{hottheorem}. 
 
\section{Proof of   Theorem \ref{basistilingtheorem}}

\begin{proof}[\bf Proof of Theorem \ref{basistilingtheorem}] Assume that $\mathcal G(|E|^{-1/2}1_E, A, B)$ is an orthonormal  basis for $L^2(\Bbb Z_p^d)$. Without loss of generality, we assume that  ${\bf 0} \in A$. Then 
\begin{equation} \label{basisprep} \sum_{x \in E} \chi(x \cdot (b-b'))=0 \ \text{if} \ b \not=b', \ b, b'\in B. \end{equation} 

Similarly, 
\begin{equation} \label{tilingprep} \sum_{x \in \Bbb Z_p^d} 1_E(x-a) 1_E(x-a')=0 \ \text{if} \ a \not=a', \ a, a'\in A \end{equation}

The equation (\ref{basisprep}) implies that $|E| \ge |B|$, while the equation (\ref{tilingprep}) implies that $E$ is a packing with $A$, thus  $|E| \leq \frac{p^d}{|A|}$. Combined with the fact that $|A||B|=p^d$ by Lemma \ref{basiccount}, we see that $|E|=|B|=\frac{p^d}{|A|}$. Invoking (\ref{basisprep}) once again, we see that $B$ must be a spectrum for $E$. Invoking (\ref{tilingprep}) we see that $A$ must be a tiling set for $E$. The conclusion of Theorem \ref{basistilingtheorem} for one direction follows.  

For the proof of the converse,  assume that $(E,A)$ is a tiling pair and $(E,B)$ is a spectral pair.  Then $|E||A|=p^d$ and $|E|=|B|$. Therefore, $|A||B|=p^d$. The orthogonality of the functions in $\{|E|^{-1/2}1_E(x-a) \chi(x\cdot b)\}_{a\in A, b\in B}$ holds by tiling property of $(E,A)$ and spectral property of $(E,B)$. 
This with the cardinality  condition $|A||B|=p^d$ completes the proof of the theorem. 
 \end{proof}

\section{Proof of   Theorem \ref{Liu-Wang_conjecture}}

    Let $w:\Bbb Z_p^d\to [0,\infty)$   and $B\subseteq \Bbb Z_p^d$.   We say  $B$ is a  spectrum for $L^2(w)$ if  $\{\chi(x\cdot b)\}_{b\in B}$ is an orthogonal basis for $L^2(w)$. That is 
  
\begin{itemize}
\item[(1)]
 the orthogonality holds: 
  $$\sum_{x\in \Bbb Z_p^d} \chi(x\cdot(b-b\rq{})) w(x) = 0 \quad \quad \forall \ b\neq b\rq{} ,\ b, b'\in B ,$$
\item[(2)]  
     $\{\chi(x\cdot b)\}_{b\in B}$ is complete in $L^2(w)$:  for any $f:\Bbb Z_p^d\to \Bbb C$ there exist  complex numbers $\{c_b\}_{b\in B}$ such that 
  
  $$ f(x) =\sum_{b\in B} c_b \chi(x\cdot b), \ \forall \ x\in \Bbb Z_p^d.$$
\end{itemize} 

  In this case, we say $B$ is a spectrum for $L^2(w)$. Notice if $w>0$ everywhere and $B$ is spectrum for $L^2(w)$,  then we must have $B=\Bbb  Z_p^d$.

 \begin{lemma}\label{square sum} Let $w: \Bbb Z_p^d\to [0,\infty)$. Suppose  that   $\sum_{x\in \Bbb Z_p^d} w(x)= 1$.  If   $B$ is a spectrum for $L^2(w)$, then  
  \begin{align}\label{square sum}
  \sum_{b\in B} |\widehat w(x-b)|^2 = p^{-2d} \quad \forall \ x\in \Bbb Z_p^d . 
  \end{align}
 \end{lemma} 
 \begin{proof}
 By the definition of the Fourier transform on $\Bbb Z_p^d$  we can write the following. 
 \begin{align}\label{FT} 
 \sum_{b\in B} |\hat w(x-b)|^2 &= \sum_{b\in B} \left|p^{-d}\sum_{m\in \Bbb Z_p^d}  w(m) \chi(-m\cdot (x-b))\right|^2 \\
 & =p^{-2d} \sum_{b\in B} \left| \sum_{m\in \Bbb Z_p^d}  \chi(-m\cdot x) \chi(m\cdot b) w(m)  \right|^2\\
 &=   p^{-2d} \sum_{b\in B} \left| \langle \chi_b, \chi_x\rangle_{L^2(w)} \right|^2  \ .
 \end{align} 
 Here, $\chi_b(t) := \chi(t\cdot b)= e^{2\pi i \frac{t\cdot b}{p}}$. 
 
 By the assumption that $B$ is a spectrum for $L^2(w)$ and  $\sum_{x\in \Bbb Z_p^d}w(x)=1$, the set  $\{\chi_b\}_{b\in B}$ is an orthonormal basis for $L^2(w)$. Therefore,  we continue as follows: 
 
 \begin{align}\label{FT} 
 \sum_{b\in B} |\hat w(x-b)|^2 =   p^{-2d}  \| \chi_x\|_{L^2(w)}^2  
   = p^{-2d} \sum_{m\in \Bbb Z_p^d} w(m) = p^{-2d}   ~. 
 \end{align} 
 This completes the proof of  the lemma. 
 \end{proof}

  \begin{lemma}[Key lemma]\label{constant measure}
Let  $w:\Bbb Z_p^d\to [0,\infty)$  with $\text{supp}(w)=E$. Assume that   
 $B$ is a spectrum for $L^2(w)$.  Then  $w=c1_E$ for some  constant $c>0$. 
 \end{lemma}
 

   \begin{proof} Without loss of generality we assume that $\sum_{m\in E} w(m)=1$. Then by Lemma \ref{square sum}  we have 
 
 $$\sum_{b\in B} |\widehat w(x-b)|^2 = p^{-2d} \quad \forall \ x\in \Bbb Z_p^d .$$
 By  summing  both sides of the equality  over $x\in \Bbb Z_p^d$,  we deduce the following. 
  
 \begin{align*}
 p^{-d} & = \sum_{x\in \Bbb Z_p^d} \sum_{b\in B} |\hat w(x-b)|^2 \\
 &=   \sum_{b\in B} \sum_{x\in \Bbb Z_p^d} |\hat w(x)|^2  \\
 & =p^{-d} |B| \sum_{m\in \Bbb Z_p^d} w(m)^2 \quad \quad \text{(by Plancherel identity (\ref{Plancherel identity}))} .
 \end{align*}
 This implies that  $\sum_{m\in \Bbb Z_p^d} w(m)^2= |B|^{-1}$. 
 To prove  $w$ is a constant function, we shall continue as follows: 
 \begin{align*}
 1 &= \sum_{m\in E} w(m)\\
 & \leq |E|^{1/2} (\sum_{m\in \Bbb Z_p^d}  w(m)^2)^{1/2} \quad \quad \text{(by Cauchy-Schwartz inequality)}
 \\
 &=  |E|^{1/2} |B|^{-1/2} =1.
 \end{align*}
 
 The result yields and $w=c1_E$ for some $c>0$. 
 
 \end{proof}

 Now we are ready to state the proof of Theorem \ref{Liu-Wang_conjecture}.

 \begin{proof}[Proof of Theorem \ref{Liu-Wang_conjecture}] 
 
 Assume that  $\mathcal G(g, A, B)$ is a Gabor orthonormal basis for $L^2(\Bbb Z_p^d)$.  Then we have 
 
 $$\sum_{x\in \Bbb Z_p^d} |g(x)|^2 \chi((b-b\rq{})x) = 0 \  \ b\neq b\rq{}, b,b' \in B.$$
 
With the assumption that $|E|=|B|$, the preceding equation implies that $\{\chi_b\}_{b\in B}$ is an orthogonal basis for $L^2(w)$ where $w=|g|^2$. By Lemma \ref{constant measure} the function   $w$ must be constant on its support.   This forces that   $|g| = |E|^{-1/2}1_E$, where $E$ is the support of $g$. To prove $(E,A)$ is a tiling pair, note that the translations of $E$ by elements of $A$ cover the whole space. The disjointness of the translations is a direct result of the cardinality. Indeed, $|A||E|=|A| |B|= p^d$, and this implies that $\sum_{a\in A} 1_E(x-a)=1$ for all $x\in \Bbb Z_p^d$, thus the proof is completed. 
\medskip 

 To prove the converse, notice 
the assumption $(ii)$ implies $|E|=|B|$. By $(iii)$ and $(ii)$,  for any $(a,b)\neq (a\rq{},b\rq{})$    we get  

 $$\sum_{x\in \Bbb Z_p^d} g(x-a)\overline{g(x-a\rq{})}   \chi((b-b\rq{})x)= 0 .$$
 
Thus $\mathcal G(g, A, B)$ is an orthonormal system in $L^2(\Bbb Z_p^d)$. The completeness is deduced from  the cardinality, since by the tiling property $|A||E|=p^d$ and by the spectral property $|B|=|E|$. 
 \end{proof}

We conclude this section by an example of a Gabor window which is not an indicator function and its support $E$  has size larger than $|B|$.

\begin{example}\label{example_of_graph} Let $F:=\{(t,t^2): t\in \Bbb Z_p\}$. $F$ is graph of a function, thus it tiles by the subgroup $\{0\}\times \Bbb Z_p$, so it is  spectral (\cite{IMP15}).  Assume that $A$ is the tiling pair and $B$ is the spectrum for $F$ with $|A|=|B|=p$. Then 
$\mathcal G(|F|^{-1/2}1_F, A, B)$ is an orthonormal basis for $L^2(\Bbb Z_p^2)$.  Let $f=|F|^{-1/2}1_F$ and   let 
  $g= \widehat{f}$ be the Fourier transform of $f$. Then by Lemma \ref{Fourier transform of G-OGB}, $\mathcal G(g, B, -A)$ is an orthonormal basis for $L^2(\Bbb Z_p^2)$.  A direct calculation shows that  $|E|=p^2-p+1$ where  $E=\text{supp}(g)$. Indeed, $|E|>p$ and $g$ is not an indicator function. Moreover, the conclusions (ii) and (iii) of Theorem \ref{Liu-Wang_conjecture} also fail. 
\end{example} 

 In  Section \ref{counter-example}  we shall present an example of a Gabor window $g$ where neither $g$ nor $\hat g$ is an indicator function. 

\section{Proof of  Theorem \ref{positive window}} 

In this section we prove Theorem \ref{positive window}.

 \begin{proof}  We shall first prove (iii). 
 Assume that $\mathcal G(g, A, B)$ is a Gabor orthonormal  basis.  Then by the orthogonality, for any $a, a\rq{}\in A$, we have 
  $$\sum_{x\in \Bbb Z_p^d} g(x-a) g(x-a\rq{}) =\sum_{E+a \cap E+a\rq{}} g(x-a) g(x-a\rq{})  =  0.  $$ 
  Since $g$ is positive, then we must have $g(x-a)g(x-a\rq{}) = 0$ for all $x\in E+a \cap E+a\rq{}$. This only happens if $E+a \cap E+a\rq{} = \emptyset$ since the functions $g(x-a)$ and $g(x-a\rq{})$  have support in $E+a$ and $E+a\rq{}$, respectively.  This proves the packing property for  $E$. 
  To complete the proof of  (iii), assume that $W$ is a subset of $\Bbb Z_p^d$ such that $W\cap E+a=\emptyset$   for all $a\in A$. We show $W$ must be empty. To this end, note that $\langle 1_W, g(x-a) e_b\rangle = 0$ by the assumption on $W$. Thus,  by the completeness of the Gabor system, $1_W$ must be zero and $W$ must be empty. This completes the proof of (iii). 
  
  \vskip.125in
   
  Next we prove (i) and (ii), simultaneously. 
   Without loss of generality, we assume ${\bf 0}\in A$. Then  by  appealing to the orthogonality of the Gabor family once again, we have 
  $$\sum_{x\in \Bbb Z_p^d} g(x)^2 \chi(x(b-b\rq{})) = 0 \quad \forall b, b\rq{}\in B, b\neq b\rq{} $$

   By (iii), $E$ is a tiling set with respect to the $A$-translations. Thus $|E||A|= p^d$. On the other hand we know $|A||B|=p^d$. Thus, $|E|=|B|$.  This, along the   orthogonality of     $\{\chi(bx): b\in B\}$ in $L^2(g^2)$, implies that  $B$ is  a spectrum for the  measure $g^2$ (\cite{IMP15}). That is  $\{\chi(bx): b\in B\}$ is an orthogonal basis for $L^2(g^2)$.  By Lemma \ref{constant measure},    $g$ has to be constant on its support. This proves (i), thus (ii). 
   
   

   
   
   
   
   
   
   
   The converse of  the theorem is due Theorem \ref{Liu-Wang_conjecture}. 
      
  \end{proof}
  
\section{Proof of Theorem \ref{counter example}}\label{counter-example}

 In this section we provide two proofs for Theorem  \ref{counter example}. 

\begin{proof} We first prove the theorem in dimension $d=2$.  Indeed, we show that there is a function $g\in L^2(\Bbb Z_p^2)$ for which the   assumption of the theorem as well as   the properties (1) and (2)  hold true. 

\medskip 

{\it Proof 1}: 
Let $f\in L^2(\Bbb Z_p)$ such 

\medskip

 \begin{align}\label{gaussian summation}
  \widehat{f}(m):=\begin{cases}
\frac{1}{\sqrt{p}} \sum_{t\in \Bbb Z_p} \chi(-mt^2) & m\neq 0\\
1 & \text{otherwise }
\end{cases}
\end{align}

Then  $|\hat f(m)|=1$ for all $m\in \Bbb Z_p$,   
%
(see e.g. \cite{Stein_Shakarchi_book}). 
By the inverse Fourier formula we have 
 \begin{align*} 
 f(x) &=1+\frac{1}{\sqrt{p}} \sum_{m \not=0} \sum_{t\in \Bbb Z_p} \chi(m(x-t^2))\\ 
 &=1-\sqrt{p}+\frac{1}{\sqrt{p}} \sum_{m\in \Bbb Z_p}\sum_{t\in \Bbb Z_p} \chi(m(x-t^2)) . 
 \end{align*}
Or,  
$$
f(x) = \begin{cases} 
1 & \ {\text {if}}\  \ x=0 \\
 1+\sqrt{p}  &\  {\text {if}} \ \  \ x\neq 0  \  \text{and} \ xRp  \\
1-\sqrt{p}  & \ {\text {if}} \ \  \ x\neq 0  \ \text{and} \  xNp .
\end{cases} 
$$
(Here, by $xRp$ we mean  $x$ is quadratic residue  modulo  $p$, and by  $xNp$ we mean  $x$ is qudratic nonresidue modulo  $p$.) 
Let $h\in L^2(\Bbb Z_p)$ such that for some constant $c>0$,  $|h(x)|= c$ for all $x\in \Bbb Z_p$.  
We define $g$ as a product of two functions  of single variable:  
\begin{align}\label{product of two functions} 
g(x_1,x_2) := f(x_1)  h(x_2)\quad (x_1,x_2)\in \Bbb Z_p\times \Bbb Z_p.
\end{align}
We shall choose   $c>0$ such that $\|g\|_2=1$.   
For the lattices 
  $A:=\{(a,0): a\in \Bbb Z_p\}$ and   $B:=\{(0,b), b\in \Bbb Z_p\}$ it can   easily be checked that   $\mathcal G(g, A, B)$ is an orthonormal set in $L^2(\Bbb Z_p^2)$. Indeed, let $(a,0)\in A$ and $(0,b_1), (0,b_2)\in B$. Then 
  \begin{align}\label{orthogonality} 
  \sum_{x_1,x_2} g(x_1-a, x_2) \overline{g(x_1,x_2)} \chi(x_2(b_1-b_2)) = c^2\left(\sum_{x_1} f(x_1-a) \overline{f(x_1)}\right)\sum_{x_2} \chi(x_2(b_1-b_2)). 
  \end{align} 
The second sum on the right equals to zero when $b_1\neq b_2$. When  $a\neq 0$,   then the sum inside the parenthesis is   zero by an application of the  Plancherel theorem and the fact  that $|\hat f(m)| = 1$ for all $m\in\Bbb Z_p$. This proves the orthogonality of the Gabor family $\mathcal G(g, A, B)$  in $L^2(\Bbb Z_p^2)$. The 
  completeness follows   by an cardinality argument.

     Notice,  in above  ${\text supp}(g)=\Bbb Z_p^2$ and 
  neither $|g|$ nor $|\hat g|$ is constant.    This completes the  proof for the existence of $g$  with desired properties in dimension $d=2$. For higher dimensions we shall continue as follows: 
  
  \medskip

  Define  $\tilde g(x_1, \cdots, x _d) := h_1(x_1)h_2(x_2) \cdots h_d(x_d)$,  where $h_1$ is the function $f$ in above, and for each $2\leq i\leq d$ we have  $|h_i(m)|= c_i$, $c_i>0$.  Take  $\tilde A:=\{(a,0,\cdots, 0): a\in \Bbb Z_p\}$ and $\tilde B:=\{(0,b_1, \cdots, b_{d-1}), b_i\in \Bbb Z_p\}$. The orthogonality of $\mathcal G(\tilde g, \tilde A, \tilde B)$ can be obtained in the same way as in (\ref{orthogonality}) and the following argument. The completeness holds by the cardinality. 
  
  \medskip 
  
  {\it Proof 2.} Here, we replace the Gaussian summation (\ref{gaussian summation}) above and function $f$ by the following function. Take $f\in L^2(\Bbb Z_p^2)$ such that $\hat f(m)=1$ for all $m\neq 0$ and $\hat f(0)=-1$. Then it is clear that $|\hat f(m)|=1$, thus the orthogonality of the Gabor set $\mathcal G(g, A, B)$ holds for same $A$ and $B$ as above. Then, by the inverse of Fourier transform,  we have $f(x)=p-2$ if $x\neq 0$ and $f(0)=-2$. Note that  neither $\hat f$ nor $f$ is constant multiple of a characteristic function, therefore the    function  $g$, constructed as in (\ref{product of two functions}), is a Gabor window function which satisfies the properties (1) and (2).  For the higher dimensions, we repeat the argument as in Proof 1. \end{proof} 
 
\section{Proof of Theorem \ref{Characterization of Gabor window for lattices}} 
 
 This section provides a proof for Theorem \ref{Characterization of Gabor window for lattices}. 

\begin{proof}  For the sake of simplicity, we  prove the theorem only for the case $d=2$. The proof for  higher dimensions is similar. 

$``\Longrightarrow"$ \ For the  proof of (a),  let $a\neq 0$. Then 
by  mutual  orthogonality of    $\mathcal G(g, A, B)$ in $L^2(\Bbb Z_p^2)$, for any $b\in \Bbb Z_p$ we  get
$$\sum_{x_2\in \Bbb Z_p}\left(\sum_{x_1\in \Bbb Z_p} g(x_1-a, x_2) \overline{g(x_1, x_2)}\right) \chi(x_2b)= 0  ~ .$$

  The uniqueness of the Fourier transform implies that  for all $a\neq 0$ we  have 

$$\sum_{x_1\in \Bbb Z_p} g(x_1-a, x_2) \overline{g(x_1, x_2)}  = 0  \quad  \forall x_2\in \Bbb Z_p .$$
If we let $g_y:\Bbb Z_p\to \Bbb C$ define $g_y(x):= g(x,y)$, then we can rewrite the preceding equation as 
$$\sum_{x_1\in \Bbb Z_p} g_{x_2}(x_1-a) \overline{g_{x_2}(x_1)}  = 0  \quad  \forall x_2\in \Bbb Z_p. \quad (\ast)$$

By applying the  Parseval identity to the equation  $(\ast)$ in variable $x_1$, for any $x_2\in \Bbb Z_p$   we obtain 

\begin{align}\label{constant fourier}
\sum_{m\in \Bbb Z_p} |\hat g_{x_2}(m)|^2 \chi(-am)  = 0 \quad \forall \ a\neq 0 .
\end{align}

The relation (\ref{constant fourier}) 
 forces $|\hat g_{x_2}|$ to be a  constant function. Or equivalently, 
 $$|\sum_{x_1} g(x_1,x_2) \chi(-x_1\cdot m)| = \text{constant} \quad \forall \ m \in \Bbb Z_p ,$$
 
hence  the proof of (a) is completed. (Notice this  constant  depends only on  $x_2$.) 
 
\vskip.125in 

To prove (b), let  $b\neq 0$. The orthogonality assumption implies that 
$$\sum_{(x_1,x_2)\in \Bbb Z_p^2} |g(x_1,x_2)|^2 \chi(-bx_2) = 0 .$$   
Thus for $G(x_2) : = \sum_{x_1} |g|^2(x_1,x_2)$, the existing equation means that  $\hat G(b)=0$ for all $b\neq 0$. This implies that 
  $G$ is constant,  indeed, $G(x)= \hat G(0)$, $x\in\Bbb Z_p$, and hence (b) follows. 

\vskip.125in

$``\Longleftarrow":$ \ To prove the converse, assume that the conditions (a) and (b) hold. By   (b)   we have 

 $$\sum_{x_1} g(x_1,x_2) \overline{ g(x_1,x_2)} =``\text{constant in} \  x_2".$$ 
 
Then for any $b\neq 0$ 
  $$\sum_{x_2} \sum_{x_1} g(x_1,x_2) \overline{ g(x_1,x_2)}\chi(bx_2) = \text{(constant)}    \sum_{x_2}  \chi(bx_2)   = 0 .$$
 
This is equivalent to saying that 
\begin{align}\label{eq:1}
\langle g, \chi_{(0,b)} g\rangle_{L^2(\Bbb Z_p^2)} = 0 \quad \forall b\neq 0 . 
\end{align}
Here,  $ \chi_{(m,n)} f(x,y) := f(x,y) \chi(mx+ny)$. 
\vskip.125in

On the other hand, (a) holds. Then for any 
   $x_2\in \Bbb Z_p$  
$$|\hat g_{x_2}(m)| = c \quad \forall  m\in \Bbb Z_p, $$ 
where $c:=c({x_2})$ is a constant depending on $x_2$. 
Thus for $a\neq 0$ 
\begin{align*} 
  \sum_m |\hat g_{x_2}(m)|^2 \chi(-am) = 0 .
  \end{align*} 
This is equivalent to  say that 
  $\langle \hat g_{x_2}, \widehat{\tau_a g_{x_2}}\rangle_{L^2(\Bbb Z_p)} = 0$, 
  where $\tau_a f(x) := f(x-a)$. 
  Thus, by  the Parseval identity  
    $$\langle   g_{x_2},  \tau_a g_{x_2}\rangle_{L^2(\Bbb Z_p)} = 0 \quad \text{for} \ a\neq 0.  $$
  This immediately leads to 
 \begin{align}\label{eq:2}\langle   g,  \chi_{(0,b)}\tau_{(a,0)} g\rangle_{L^2(\Bbb Z_p^2)} = 0 \quad \forall a\neq 0 . 
 \end{align}
  
  From the  orthogonality relations in (\ref{eq:1}) and (\ref{eq:2})  we deduce that the set $\mathcal G(g, A, B)$ is mutual  orthogonal in $L^2(\Bbb Z_p^2)$. The completeness  is a conclusion of the cardinality. 
\end{proof}

\begin{example}
Here we construct an example of Gabor window satisfying (a) and (b) in Theorem  \ref{Characterization of Gabor window for lattices} and $g$ is not an indicator function. 
%
  Define   $g:\Bbb Z_p^2\to \Bbb C$ as follows: For any $(k,r)\in \Bbb Z_p^2$

  $$ 
  g(k,r): = 
  \begin{cases} 
  1 & \text{if}  \ \  k=0 \\
  1+ \sqrt{p}  & \text{if}  \ \  k\neq 0 \  \ \text{and}  \ \  kRp \\
    1 -\sqrt{p}  & \text{if}  \ \  k\neq 0 \ \  \text{and} \ \  kNp .
  \end{cases}
    $$
       
 (In above, by $kRp$ and $kNp$ we mean  $k$ is quadratic residue or  non-residue (mod $p$), respectively.) 
     
   The definition indicates that the all rows of the matrix $[g(k,r)]_{p\times p}$   are the same, where  $(k,r)\in \Bbb Z_p\times \Bbb Z_p$ and $g(k,r)$ indicates the entry in $k$-th row and $r$-th column. Therefore, 
 $g$ satisfies both conditions (a) and (b) in Theorem \ref{Characterization of Gabor window for lattices}. Moreover, 
 for any $(m,n)\in \Bbb Z_p^2$

  $$ 
  \widehat g(m,n): = 
  \begin{cases} 
  1 & \text{if}  \ \  m=0, n=0 \\
   p^{-1/2} \sum_{t\in \Bbb Z_p} \chi(-mt^2) & \text{if}  \ \  m\neq 0, n=0  \\ 
  0 &  \text{elsewhere.}  
  \end{cases}  
    $$
     
     Indeed, it is easy to see that $\widehat g(0,0) = 1$. 
For $m\neq 0$, by the definition of Fourier transform,  a direct calculation shows that 
     
     $$\widehat g(m,0) = p^{-1/2} \sum_{k\in \Bbb Z_p} |Z(t^2-k)| \chi(-mk),  $$
     
     %
     where $|Z(t^2-k)|=1$ if $k=0$,     $|Z(t^2-k)|=2$ if $kRp$ and $k\neq 0$,  and  $|Z(t^2-k)|=0$ if $kNp$ and $k\neq 0$. 
     
Notice that neither $g$ nor $\widehat g$ is an indicator function. 
\end{example} 
 \medskip 
 
 We conclude the paper with the following open problems: 
 
 \medskip 
 
\begin{question}
 Is there any Gabor window $g\in L^2(\Bbb Z_p^d)$ with support $E$  for which  the all following properties hold simultaneously?

 \begin{itemize} 
 \item  $g$ is  not  in form of product of functions of single variable. 
 \item $g$ is not positive. 
 \item $|g|\neq c1_E$, i.e.  is not multiple of any characteristic function. 
 \item $|\hat g|\neq c1_F$, i.e.  is not multiple of any characteristic function. 
 \item $|E|\neq |B|$  (thus $E$ does not tile). 
 \item $E$ does not tile (thus $|E|\neq |B|$). 
 \end{itemize}
 \end{question}

\begin{problem} Classify all 
 Gabor orthonormal bases of the form 
 
  $$\mathcal G(g, S):= \{  g(x-a) \chi(x\cdot b): \  (a,b)\in S\} $$  
  and  $S \subset \Bbb Z_p^{2d}$, $d\geq 1$. 
  \end{problem} 
 
  As mentioned earlier,   Lemma \ref{basiccount} provides a characterization of Gabor orthonormal bases for dimension $d=1$ when $S$ is separable, i.e.  $S=A\times B$.

\vskip.125in 

\end{document}